\documentclass[12pt]{article}
\usepackage{latexsym,color,amsmath,amsthm,amssymb,amscd,amsfonts}
\usepackage{tikz}
\usetikzlibrary{arrows}
\usepackage{scalefnt}

\usepackage{float}
\usepackage{graphicx} 

\usepackage{amsopn}
\usepackage{relsize}

\newtheoremstyle{nonum}{}{}{\itshape}{}{\bfseries}{.}{ }{\thmnote{#3}}

\newtheorem{thm}{Theorem}[section]

\newtheorem{cor}[thm]{Corollary}
\newtheorem{lem}[thm]{Lemma}
\newtheorem{prop}[thm]{Proposition}
\newtheorem{rem}[thm]{Remark}
\newtheorem{conj}[thm]{Conjecture}

\theoremstyle{nonum}

\newcommand{\R}{\mathbb R}
\newcommand{\RR}{\mathbb R}

\def\K{{\cal K}}

\def\Vol{{\rm Vol}}

\newcommand{\iprod}[2]{\langle #1,#2 \rangle} 

\def\K{{\cal K}}

\begin{document}
\title{Remarks about Mixed Discriminants and Volumes}
\date{}
\author{S. Artstein-Avidan, \ \ D. Florentin, \ \ Y. Ostrover}
\maketitle
\begin{abstract}
In this note we prove certain inequalities for mixed discriminants of positive semi-definite
matrices, and mixed volumes of compact convex sets in ${\mathbb R}^n$.
Moreover, we discuss how the latter are related to the 
 monotonicity  of an 
information functional on the class of convex bodies, which is a geometric analogue of  the classical  Fisher information.

%
\end{abstract}

\section{Introduction and Results}

Starting from the seminal works of A. D. Aleksandrov, the theory of mixed discriminants 
and volumes  serves as a powerful tool for studying  various quantities associated with 
convex bodies, such as volume, surface area, and mean width.
In addition to  their significant role in convex geometry, inequalities emanating from this 
theory - the most famous of which is probably the Alexandrov--Fenchel inequality - have numerous 
applications and deep connections to various fields, such as differential and algebraic
geometry, probability theory, combinatorics, and more.  We refer the reader to~\cite{Gar, Sch}, 
and the references therein, for a more detailed exposition of this subject.

In this paper  we prove some inequalities for mixed discriminants of positive semi-definite
matrices, and mixed volumes of compact convex sets in ${\mathbb R}^n$, denoted
by $D(A_1,\dots,A_n)$ and $V(K_1,\dots,K_n)$ respectively (the precise 
definitions will be given in the following sections). 
Our work is partially motivated by a result of Hug and Schneider regarding a certain inequality for mixed
volumes of zonoids (Theorem 2 in~\cite{HS}), which is conjectured to hold (ibid., page 2643) 
for arbitrary convex bodies (cf.~inequality (16) in~\cite{BW}).
Our first result is the following simple observation regarding  mixed discriminants which seems to have been
overlooked in the literature. 

\begin{thm} \label{Thm-about-mixed-discriminants}
For any positive semi-definite $n \times n$ matrices $A_1,A_2,A_3$ one has:
\begin{equation}\label{Ineq_Disc}
D(A_1,A_3[n-1])  D(A_2,A_3[n-1])\ge \frac{n-1}{n}  \, D(A_1,A_2,A_3[n-2])  D(A_3[n]).
\end{equation}
Moreover, equality holds if and only if one of the following three cases occurs: (i)
 $A_3$ is invertible and $A_1A_3^{-1}A_2=0$,  (ii)
 $A_3$ is of rank at most $n-2$, (iii) $A_3$ is of rank $n-1$ and either $A_1$ or $A_2$ satisfies ${\rm Im}(A_i) \subset {\rm Im}(A_3)$.
\end{thm}

Here, and in the following, the abbreviation $A[i]$ stands for $i$ copies of the object $A$.
The analogue of inequality~$(\ref{Ineq_Disc})$ for mixed volumes of convex bodies is
\begin{equation}  \label{MV-general-(false)-statement} V(K,A[n-1])V(T,A[n-1])\ge \frac{n-1}{n}
V(K,T,A[n-2])V(A[n]). \end{equation}
Note that if two of the bodies coincide, this inequality
holds even without the $\frac{n-1}{n}$ factor due to  Alexandrov--Fenchel
inequality for $K=T$, or
trivially for $A=K$ or $A=T$.
Inequality~$(\ref{MV-general-(false)-statement}$)
fails in general for $n\ge 3$ (see, e.g., Subsection~\ref{section-counter-example} below). 
However, in the case where $A=B_2^n$ is the Euclidean unit ball and $T$ is a zonoid, not only does inequality~$(\ref{MV-general-(false)-statement})$ hold, but actually 
a slightly stronger inequality is valid. Namely, 
\begin{thm} \label{Thm-about-mixed-volumes}
For every convex body $K \subset {\mathbb R}^n$ and every zonoid $Z \subset {\mathbb R}^n$
\begin{equation}\label{Ineq_Vol-A=B}
V(K,B_2^n[n-1])  V(Z,B_2^n[n-1])\ge \frac{n-1}{n}  {\frac  {\kappa^2_{n-1}} {\kappa_n  \kappa_{n-2}} } V(K,Z,B_2^n[n-2])  V(B_2^n[n]),
\end{equation}
where $\kappa_n$ stands for the volume of the $n$-dimensional Euclidean unit ball.
Moreover, equality holds if and only if $K$ and $Z$ lie in orthogonal (affine) subspaces of ${\mathbb R}^n$. 
\end{thm}
When both $K$ and $Z$ are zonoids, inequality~$(\ref{Ineq_Vol-A=B})$  was proved by  Hug and  Schneider in~\cite{HS},  (cf.~\cite{BW}), and was conjectured  to hold for  arbitrary convex bodies $K$ and $Z$.
Note that the constant ${\frac {\kappa^2_{n-1}}  {\kappa_n   \kappa_{n-2}} }$ in~$(\ref{Ineq_Vol-A=B})$
is strictly greater than one, and approaches one as $n$ tends to infinity. More precisely $1<{\frac {\kappa^2_{n-1}}  {\kappa_n  \kappa_{n-2}} } <1+\frac{1}{n-1}$. 

It turns out that  inequality~$(\ref{MV-general-(false)-statement})$  fails  for some triples $(K,T,A)$, even in the case where $K=B_2^n$ is the Euclidean unit ball and $T$ is an interval.
This case is equivalent to an inequality which was conjectured by  Giannopoulos, Hartzoulaki, and Paouris in~\cite{GHP},  and then disproved by Fradelizi, Giannopoulos, and Meyer in~\cite{FGM}, where  also a positive result was proven which gives a special case of~$(\ref{MV-general-(false)-statement})$ with different constants. 
In Subsection~\ref{section-counter-example}  we give yet another example of the failure of~$(\ref{MV-general-(false)-statement})$ when $K=B_2^n$, $T$ is an interval, and   
$A$ is a certain truncated box. Any case where~$(\ref{MV-general-(false)-statement})$ fails with $K=B_2^n$  gives a negative  answer to  the question of  the monotonicity 
of a certain geometric analogue of the Fisher information functional on the class of convex domains which was introduced in~\cite{DCT}. 
More precisely,  denote by $\K^n$ the class of 
compact convex sets in ${\mathbb R}^n$, and for $K \in \K^n$ set $I(K)=|K|/|\partial K|$, where $|K|$ stands for the volume of $K$ and $|\partial K|$ for its surface area.

The functional $I$, which can be considered as a dual  analogue of the Fisher information,  
 was introduced by Dembo, Cover, and Thomas in~\cite{DCT}. In the same paper it was asked whether $I$ satisfies a Brunn--Minkowski type inequality i.e.,  whether for any $K_1,K_2 \in \K^n$ one has $I(K_1+K_2) \geq I(K_1) + I(K_2)$, or at least whether  $I$ is monotone with respect to Minkowski addition, namely, satisfies $I(K_1+K_2) \geq I(K_1)$ for every 
$K_1,K_2 \in {\cal K}^n$.  
In~\cite{DCT} it was verified that $I$ is  monotone with respect to the addition of a Euclidean ball. This is a simple consequence of the Alexandrov--Fenchel inequality. It was also noted that without convexity the above mentioned Brunn--Minkowski type inequality cannot hold. 
In~\cite{FGM} it was shown that even for convex bodies, a counterexample to this 
 inequality exists.  In fact, the example given in~\cite{FGM} is also a counterexample for the monotonicity question above, although this was not pointed out explicitly in~\cite{FGM}.  

Our next observation regarding mixed volumes is the equivalence of the monotonicity property of $I$ with a certain inequality for mixed volumes. More precisely, 

\begin{prop}\label{Prop_Equiv-To-Info-Monoton}
Let $T\in\K^n$. The following two inequalities are equivalent:
\begin{itemize}
\item[(i)] $\forall A\in\K^n:\, V(B_2^n,A[n-1])  V(T,A[n-1])\ge \frac{n-1}{n} \, V(B_2^n,T,A[n-2])  V(A[n])$;
\item[(ii)]$\forall A\in\K^n:\, I(A+T)\ge I(A)$.
\end{itemize}
\end{prop}

In Section~\ref{Sec_Vol-n=2} we shall prove that in dimension $2$ the information functional $I$ is monotone with respect to Minkowski addition. In fact, inequality \eqref{MV-general-(false)-statement} holds for all $K,T$ and $A$. As noted above, in any other dimension $n \geq 3$ 
both inequalities in Proposition~\ref{Prop_Equiv-To-Info-Monoton} fail in general. 
In Subsection~\ref{section-counter-example} we give for any $n\ge 3 $ an explicit example of a pair of convex bodies $T,A \in {\cal K}^n$ for which the two inequalities in Proposition~\ref{Prop_Equiv-To-Info-Monoton} fail to hold.
It remains an interesting question to determine for which convex bodies
$T$ the inequality in Proposition \ref{Prop_Equiv-To-Info-Monoton} does hold, and monotonicity is satisfied (for example, the ball $B_2^n$ is such a body). 

The rest of the paper is organized as follows:  in Sections~\ref{Sec_Disc} and~\ref{Sec-MV-int-and-theorem}  we prove Theorems~\ref{Thm-about-mixed-discriminants} and~\ref{Thm-about-mixed-volumes},
respectively.  In Section~\ref{Section-relation-with-entropy} we discuss the relation between the information functional $I$ and inequality~$(\ref{MV-general-(false)-statement})$, and prove Proposition~\ref{Prop_Equiv-To-Info-Monoton}.
 Finally, in Section~\ref{Sec_Vol-n=2} we prove inequality~$(\ref{MV-general-(false)-statement})$ in the two-dimensional case.

\noindent{\bf Notations:} Throughout the text we shall use the following notations:  
By a convex body we shall mean a compact convex set with non-empty
interior. The class of convex bodies in ${\mathbb R}^n$
is denoted by $\K^n$. Given $K \in \K^n$, we denote by $h_K : {\mathbb R}^n \to {\mathbb R}$  its support function, given by 
$ h_K(u) = \sup \{ \langle x, u \rangle \ ; \  x \in  K \}$.
We set $\sigma$ to be the normalized Haar measure on the sphere $S^{n-1}\subset {\mathbb R}^n$,
and $\lambda_
{n}$ the standard $n$-dimensional Lebesgue measure.  The volume of the $n$-dimensional Euclidean unit ball is denoted by $\kappa_n$.
Finally, we denote $M^*(K):= \int_{S^{n-1}}
h_{K} \, d\sigma$.

\noindent{\bf Acknowledgments}: We thank Prof. R. Schneider for his comments on the written text. 
The first and second named authors were partially supported by ISF grant No. 247/11.  The third named author was partially supported by a Reintegration Grant SSGHD-268274 within the 7th European community framework programme, and by   ISF grant No. 1057/10. 

\section{Mixed Discriminants}\label{Sec_Disc}

Mixed discriminants were introduced by A. D. Aleksandrov as a tool to study mixed volumes
of convex sets (see e.g., \textsection 25.4 in~\cite{BZ} and the references therein).  
They  are the coefficients in the polynomial expansion
of the determinant of a sum of matrices.
More precisely, let $A_1,\ldots, A_m$ be symmetric real $n \times n$ matrices .The determinant of  the sum $\sum_{i=1}^m \lambda_i A_i$ is a homogeneous polynomial of degree $n$ in $\lambda_1,\ldots,\lambda_m$, and can be written as
\begin{equation}\label{Eq_Disc-Def}
\det \left(\sum_{i=1}^m \lambda_i A_i\right) =
\sum_{i_1,\dots,i_n=1}^m \lambda_{i_1}\cdots\lambda_{i_n} D(A_{i_1},\dots,A_{i_n})
\end{equation}
(see~\cite{BZ},  or~\textsection 2.5 in~\cite{Sch}). 
The quantity $D(A_1,\ldots,A_n)$ is called the mixed discriminant of $A_1,\ldots,A_n$.

In the following lemma we gather some basic well known facts
regarding mixed discriminants (see e.g.~\cite{Bap}). Here, $A^i$
stands for the  $i$-th column of the matrix $A$, the notation
$A \geq 0$ means that $A$ is symmetric and positive semi-definite,
and $\Pi_n$ stands for the permutation group of $n$ elements.
\begin{lem}  \label{basic-prop-of-md}
Let $A_1,\ldots, A_n$ be symmetric real $n \times n$ matrices. 
\begin{itemize}
\item[(i)]  If $A_i \geq 0$ for all $i$, then $D(A_1,\ldots, A_n) \geq 0$;
\item[(ii)]  $D(BA_1,\dots,BA_n)=\det(B)D(A_1,\dots,A_n)$, for any  $n\times n$  matrix $B$;
\item[(iii)] $D(A_1,\dots,A_n)=\frac{1}{n!}\sum_{\sigma\in \Pi_n}
\det(A_{\sigma(1)}^1,\dots,A_{\sigma(n)}^n).$
\end{itemize}
\end{lem}

Note that if $A_i=A$ for all $i$, then $D(A_1,\ldots, A_n) = \det(A)$.
We are now in a position to prove Theorem~\ref{Thm-about-mixed-discriminants}.
\begin{proof}[{\bf Proof of Theorem~\ref{Thm-about-mixed-discriminants}}]
From property $(i)$ of Lemma~\ref{basic-prop-of-md} it follows  that  inequality~$(\ref{Ineq_Disc})$ holds trivially when  $\det(A_3)=0$. Thus, we can assume without loss
of generality that $A_3$ is invertible. Hence, using property $(ii)$ of Lemma~\ref{basic-prop-of-md},  we conclude that in order to prove inequality~$(\ref{Ineq_Disc})$
it suffices to show that:
\begin{equation} \label{MD-ineq-reduced} D(X,I[n-1])D(Y,I[n-1])\ge \frac{n-1}{n}D(X,Y,I[n-2]), \end{equation}
where $X=A_3^{-1}A_1$, $Y=A_3^{-1}A_2$, and $I$ is the $n\times n$ identity matrix. By property $(iii)$ of Lemma~\ref{basic-prop-of-md} we have:
$$D(X,I[n-1])=\frac{1}{n}\sum_{i=1}^n \det(e_1,\dots,e_{i-1}, X^i,e_{i+1}, \dots,e_n)
=\frac{1}{n}\sum_{i=1}^n x_{ii}=\frac{{\rm tr}(X)}{n},$$
where  $\{e_i\}_{i=1}^n$ stands for the $i$-th  column of the identity matrix $I$, and 
\begin{eqnarray}
 D(X,Y,I[n-2])   =
\frac{1}{n(n-1)}\sum_{i\neq j}\det(Z(i,j)). \nonumber
\end{eqnarray}
where $Z(i,j)$ denotes the identity matrix with the $i^{th}$ column replaced by $X^{i}$ and the $j^{th}$ column replaced by $Y^j$. 
Separating into two sums we get 
\begin{eqnarray}
 D(X,Y,I[n-2])   & = &
 \frac{1}{n(n-1)} 
\sum_{i<j}\det(e_1,\dots,X^i,\dots,Y^j,\dots,e_n)+  \nonumber\\
&&  \frac{1}{n(n-1)} \sum_{j<i}\det(e_1,\dots,Y^j,\dots,X^i,\dots,e_n)  
\\
& = &  \frac{1}{n(n-1)}\sum_{i\neq j}(x_{ii}y_{jj}-x_{ji}y_{ij})=
\frac{1}{n(n-1)}\sum_{i,j=1}^n(x_{ii}y_{jj}-x_{ji}y_{ij}). \nonumber
\end{eqnarray}
Combining these relations we conclude that
\begin{eqnarray}
\lefteqn{D(X,I[n-1])D(Y,I[n-1])-\frac{n-1}{n} D(X,Y,I[n-2])}\nonumber\\
& = &\left(\frac{1}{n}\sum_{i=1}^n x_{ii}\right)
\left(\frac{1}{n}\sum_{j=1}^n y_{jj}\right)-\frac{n-1}{n} \cdot \frac{1}{n(n-1)}
\sum_{i,j=1}^n(x_{ii}y_{jj}-x_{ji}y_{ij}) \nonumber
\\
& = & \frac{1}{n^2}\sum_{i,j=1}^n(x_{ii}y_{jj}-x_{ii}y_{jj}+x_{ji}y_{ij}) =\frac{1}{n^2}\sum_{i,j=1}^n x_{ji}y_{ij}=\frac{{\rm tr}(XY)}{n^2}. \nonumber
\end{eqnarray}
Note that $\operatorname{tr}(XY)=\operatorname{tr}(A_3^{-1}A_1A_3^{-1}A_2)$. Moreover, it is not hard to check that the matrix $A_3^{-1}A_1A_3^{-1}$ is also symmetric and positive semi-definite. 
Inequality~$(\ref{MD-ineq-reduced})$ now immediately follows since the trace of the product of two symmetric positive semi-definite matrices is always  
non-negative\footnote{Indeed, for any two symmetric positive semi-definite matrices $A$ and $B$  one has $\operatorname{tr}(AB) =\operatorname{tr}(\sqrt{A}\sqrt{A}\sqrt{B}\sqrt{B})
=\operatorname{tr}(\sqrt{B}\sqrt{A}\sqrt{A}\sqrt{B}) =\operatorname{tr}((\sqrt{A}\sqrt{B})^*\sqrt{A}\sqrt{B}) $, which is a sum of squares.}.
Moreover, under the assumption that $A_3$ is invertible, equality in~$(\ref{MD-ineq-reduced})$ holds if and only if  $\operatorname{tr}(A_3^{-1}A_1A_3^{-1}A_2)=0$, or equivalently (as it is the product of two positive definite matrices), $A_1A_3^{-1}A_2=0$. Moreover, it follows from~\cite{Pan}  that for singular $A_3$ equality in~$(\ref{Ineq_Disc})$ holds if and only if either 
$A_3$ is of rank at most $n-2$, or $A_3$ is of rank $n-1$ and either $A_1$ or $A_2$ satisfies ${\rm Im}(A_i) \subset {\rm Im}(A_3)$.
This completes the proof of Theorem~\ref{Thm-about-mixed-discriminants}.
\end{proof}



\section{Mixed Volumes} \label{Sec-MV-int-and-theorem}

Arising from the classical works of Minkowski, Aleksandrov, Hadwiger, and
many others, mixed volumes have been studied in a variety of contexts.
In addition to having vast applications to convex geometry, mixed volumes
provide geometric techniques to study sparse systems of polynomial
equations - and thus serve as a bridge between algebraic and convex
geometry, appear as intersection numbers in tropical geometry, and can
be used as a powerful tool in combinatorics and computational geometry.
For a detailed exposition and  further information on the properties of
mixed volumes we refer the reader to Chapter 5 of~\cite{Sch}.

A classical result due to Minkowski states that the volume of a
linear combination $\sum_{i=1}^m \lambda_i K_i$ of convex bodies
$K_i$ is a homogeneous polynomial of degree $n$  in  $\lambda_i \ge 0$, 
where $A+B$ stands for Minkowski addition,  $A+B=\{a+b: a\in A,b\in B\}$.
Mixed volumes are the coefficients in this polynomial expansion.
More precisely
\begin{equation}\label{Eq_Vol-Def}
\Vol \left(\sum_{i=1}^m \lambda_i K_i\right) = 
\sum_{i_1,\dots,i_n=1}^m \lambda_{i_1}\cdots\lambda_{i_n} V(K_{i_1},\dots,K_{i_n}),
\end{equation}
where $K_i\subset\R^n$ are compact convex sets and $\lambda_i \ge 0$. 
The coefficient $V(K_{i_1},\ldots,K_{i_n})$ of the monomial $\lambda_{i_1} \cdots \lambda_{i_n}$ is called the mixed volume of $K_{i_1}, \ldots,K_{i_n}$,
and it depends only on $K_{i_1}, \ldots, K_{i_n}$ and not on any of the other bodies.
One may assume that the coefficients are  symmetric with respect to permutations of
the bodies. Mixed volumes are known to be non-negative, and are clearly translation
invariant. Moreover, they are monotone with respect to set inclusion, additive in
each argument with respect to Minkowski addition, continuous  with respect to the
Hausdorff topology, and positively homogeneous in each argument (see e.g.
Section \textsection 5.1 of~\cite{Sch}, and~\cite{Egg}, Chapter 5). 
%
We now turn to the proof of Theorem~\ref{Thm-about-mixed-volumes}.

\begin{proof}[{\bf Proof of Theorem~\ref{Thm-about-mixed-volumes}}]
Since mixed volumes are continuous with respect to the Hausdorff topology on $\K^n$, 
in order to prove inequality~$(\ref{Ineq_Vol-A=B})$ it  suffices to assume that 
 $Z$ is a zonotope, i.e., a Minkowski sum of intervals. Moreover, using the additivity, translation invariance, and positive  homogeneity of mixed volumes, 
we may further assume  that $Z$ is the interval $Z=[0,u]$, for some  $u\in S^{n-1}$. 

%

Let $u \in S^{n-1}$. We denote by $\sigma'$ the normalized Haar measure on
the sphere $S^{n-2}$ which we identify with $S^{n-1}\cap u^\perp$,
by $\nu$ the $(n-1)$-dimensional mixed volume functional in
${\mathbb R}^{n-1}$, and by $K|_u$ the orthogonal projection of $K$ onto
$u^\perp$. It is well known (see e.g.,  equation (A.43) in~\cite{Gar})
that: $$V([0,u],K_2,\dots,K_n) = \frac{1}{n} \nu(K_2|_u,\dots,K_n|_u).$$
Combining this with the standard integral representation of
quermassintegrals (that is, mixed volumes where only two bodies
are mixed - some body $K$ and the Euclidean ball $B_2^n$), see
(5.1.18) in \textsection 5.3 of~\cite{Sch}, we conclude that
$$
V(K,B^n[n-1])=\kappa_n \int_{S^{n-1}}
h_{K} \, d\sigma,\qquad V(Z,B^n[n-1])=\frac{\kappa_{n-1}}{n},$$
$$
V(K,Z,B^n[n-2])=\frac{1}{n}\nu(K|_u,B^{n-1}[n-2])= 
\frac{\kappa_{n-1}}{n}\int_{S^{n-1}\cap u^\perp} h_{K} \, d\sigma',$$
where in the last equality we have used the fact that $h_K=h_{K\vert_u}$ on
$u^\perp$. With these equalities at our disposal,  inequality~(\ref{Ineq_Vol-A=B})
reduces to showing
\begin{equation} \label{equation-about-int-in-MV-proof} M^*(K):= \int_{S^{n-1}}
h_{K} \, d\sigma \ge C_n
\int_{S^{n-1} \cap u^\perp} h_{K} \, d\sigma',
\end{equation}
where
the constant $C_n$ is the same as in (\ref{Ineq_Vol-A=B}), i.e., $C_n=
{\frac {n-1} {n}} \, {\frac {\kappa_{n-1}^2} {\kappa_n \kappa_{n-2}} }$.

Denote by $\Pi_u$
 the reflection operator with respect to the hyperplane
 $u^{\perp}$, that is, 
 $\Pi_u (x) = x-2u\iprod{x}{u}$.
 Clearly 
$M^*(K)= M^*(\frac{K+\Pi_uK}{2})$. The body $\frac{K+\Pi_uK}{2}$ is called the Minkowski symmetrization of $K$ in direction $u$. It 
is also clear that $\frac{K+\Pi_uK}{2}\supseteq K|_u$. Thus we get that 
  $M^*(K)= M^*(\frac{K+\Pi_uK}{2})\ge M^*(K|_u)$, where  the second inequality
is due to inclusion. Moreover, since  $h_K=h_{K \vert_u}$ on
$u ^\perp$, the right-hand side of~$(\ref{equation-about-int-in-MV-proof})$
is the same for $K$ and for $K|_u$. Therefore, to prove
inequality~$(\ref{equation-about-int-in-MV-proof})$ it suffices to
assume that $K\subset u^\perp$. On the other hand, in the case
where $K\subset u^\perp$, the left and right-hand side of
equation~$(\ref{equation-about-int-in-MV-proof})$ are equal. Indeed,
it is well known (see e.g.~\cite{Gar}, page 404, equation (A.29))
that for a $k$-dimensional convex body in $\RR^n$ with $k<n$ one has
for any $0 \leq i \leq k$ that
\[ {\frac {1}{c_{i,k}}} V(K[i],B^k_2[k-i]) =
  {\frac {1}{c_{i,n}}} V(K[i],B^n_2[n-i]), \ {\rm where} \ c_{i,k} = \frac{\kappa_{k-i}}{\binom{k}{i}},\]
where the mixed volume functional on the left-hand side is $k$-dimensional.  
Thus, we conclude that
\begin{eqnarray*}   \int_
{S^{n-1}}h_K d \sigma  &= & \frac{1}{\kappa_n} V(K ,B_2^n[n-1]) =  \frac{1}{\kappa_n}\frac{c_{1,n}}{c_{1,n-1}} V(K ,B_2^{n-1}[n-2])  \\
&= & \frac{\kappa^2_{n-1}} {\kappa_n \kappa_{n-2}} \frac{n-1}{n}   \int_{S^{n-1}
\cap u^\perp} h_Kd\sigma'. \end{eqnarray*}
Next we  characterize the equality case. For $Z=[0,u]$, equality in
$M^*(K)\ge M^*(K|_u)$, and hence in~$(\ref{Ineq_Vol-A=B})$, holds if
and only if $K \subset x_0 + u^\perp$ for some $x_0 \in \RR^n$.
Using the additivity of inequality~$(\ref{Ineq_Vol-A=B})$ in the
parameter $Z$, we conclude that for a zonotope $Z$ (i.e., a finite
Minkowski sum of line segments), equality in~$(\ref{Ineq_Vol-A=B})$
holds if and only if, up to translations, $K \subset E$ and
$Z\subset E^{\perp}$ for some  linear subspace $E \subset {\mathbb R}^n$.
Finally, for a general zonoid $Z$ we argue as follows.  Let
$K \in \K^n$, and consider the function 
$$I_K(Z):=   V(K,B^n[n-1])  V(Z,B^n[n-1]) -  \frac{n-1}{n}
 {\frac  {\kappa^2_{n-1}} {\kappa_n  \kappa_{n-2}} }  V(K,Z,B^n[n-2])  V(B^n[n]),$$
defined on the class of zonoids in ${\mathbb R}^n$. 
Set $E_K$ to be the subspace  of ${\mathbb R}^n$ parallel to
the minimal  affine space containing $K$. In these notations
the above argument implies that $I_K([-u,u])=0$ only if
$u \in E_K^{\perp}$. Moreover, since any zonoid $Z$ is given
by $Z=\int_{S^{n-1}} [-u,u] d\mu_Z$, where $\mu_Z$ is an even
measure on $S^{n-1}$ (see~\textsection 3.5 in~\cite{Sch}), it
follows from the properties of mixed volume that
$I_K(Z)=\int_{S^{n-1}}I ([-u,u])d\mu_Z$. Hence, using the fact
that $Z \subset E_K^{\perp}$ if and only if $\mu_Z$ is supported
on $S^{n-1} \cap E_K^{\perp}$, we conclude that $I_K(Z)>0$ if and
only if  $Z \not\subset E_K^{\perp}$, and the proof is now complete. 
\end{proof}

%

\begin{rem} {\rm 
Note that if $K=Z$ in Theorem~\ref{Thm-about-mixed-volumes}, even without the assumption that $Z$ is a zonoid (i.e., for any $Z \in \K^n$),
inequality~(\ref{Ineq_Vol-A=B})
holds without the ${\frac {n-1} {n} } \, {\frac  {\kappa_{n-1}^2}  {\kappa_n \kappa_{n-2}} }$ factor, thanks to Minkowski's first
inequality  (see Equation (6.2.3) in~\cite{Sch}). }
\end{rem}

\begin{rem} {\rm 
As mentioned in the introduction, inequality~(\ref{Ineq_Vol-A=B}) is conjectured to hold for arbitrary convex bodies $K,T \in \K^n$ (see~\cite{HS}).
This conjecture can be reformulated in the language of harmonic analysis, or, more precisely, as a conjecture on the rate of decay of the coefficients of a convex function with respect to its spherical harmonic decomposition. 
More precisely, consider the following integral formula for 
$V(K,T,B[n-2])$ (see (5.3.17) in~\cite{Sch}):
$$V(K,T,B[n-2])=\kappa_n\int_{S^{n-1}}h_K(u) \Bigl (h_T(u)+\frac{1}
{n-1}\triangle_S h_T(u) \Bigr ) d\sigma(u),$$
where $\triangle_S$ stands for the spherical Laplace operator on $S^{n-1}$.
Combining this with the fact  that $V(K,B,\dots,B)= \kappa_n\int_{S^{n-1}}h_K d \sigma$ (see (5.3.12) in~\cite{Sch}), 
 we conclude that
for $K,T \in K^n$ inequality~(\ref{Ineq_Vol-A=B}) can be written as:
\begin{equation} \label{eq-form-of-ineq-about-KTB} 
\int_{S^{n-1}}h_K \, d\sigma \int_{S^{n-1}}h_T \, d\sigma \geq C_n \int_{S^{n-1}}h_K(u) \Bigl (h_T(u)+\frac{1}
{n-1}\triangle_S h_T(u) \Bigr ) d\sigma(u), \end{equation}
where $C_n =  {\frac {n-1} {n} } \, {\frac {\kappa_{n-1}^2} {\kappa_{n} \kappa_{n-2}} }$. 
Next, consider the harmonic expansion of the support functions  $h_K$ and $h_T$, 
given by $$h_K = \sum_{m=0}^{\infty} \sum_{l=0}^{N(m,n)}  k_{m,l} Y^l_m, \ {\rm and} \ \ h_T = \sum_{m=0}^{\infty} \sum_{l=0}^{N(m,n)}  t_{m,l}  Y^l_m,$$ where $\{ Y^l_m\}_{l=0}^{N(m,n)}$ stands for a basis for the $m$-eigenspace of the spherical Laplace operator $\triangle_S$ on the space $L_2(S^{n-1})$  with eigenvalue $\triangle_S Y^l_m = -m(m+n-2)Y^l_m$ (see the Appendix of~\cite{Sch} for more details). 
It is well known 
 that $ k_{0,0}=   \int_{S^{n-1}}h_K d\sigma$, and similarly $t_{0,0} =   \int_{S^{n-1}}h_T d\sigma$.
Thus, we conclude that  inequality~(\ref{Ineq_Vol-A=B}) is equivalent to 
\begin{equation} \label{eq-for-spherical-harmonics1}
k_{0,0}t_{0,0} \geq C_n \sum_{m=0}^{\infty}  { \frac {(1-m)(n+m-1)} {n-1} }  \sum_{l=0}^{N(m,n)}  k_{m,l} t_{m,l},
\end{equation}
which after a suitable rearrangement can be written as
\begin{equation} \label{eq-for-spherical-harmonics2}
k_{0,0} t_{0,0} \geq {\frac {C_n} {1-C_n}} \sum_{m=1}^{\infty}   { \frac {(1-m)(n+m-1)} {n-1} } \sum_{l=0}^{N(m,n)}      k_{m,l} t_{m,l}.
\end{equation}
We thus arrive at the following conjecture
\begin{conj}
Let $h_1, h_2: S^{n-1}\to \RR$ be two convex functions (that is, their homogeneous extension to $\RR^n$ is assumed convex). Let their decomposition in the basis of spherical harmonics be given by 
\[ h_1 = \sum_{m=0}^{\infty} \sum_{l=0}^{N(m,n)}  a_{m,l} Y^l_m, \ {\rm and} \ \ h_2 = \sum_{m=0}^{\infty} \sum_{l=0}^{N(m,n)}  b_{m,l}  Y^l_m. \]
Then 
\begin{equation}
a_{0,0} b_{0,0} \geq { D_n} \sum_{m=1}^{\infty}   { \frac {(1-m)(n+m-1)} {n-1} } \sum_{l=0}^{N(m,n)}      a_{m,l} b_{m,l}, 
\end{equation}
where $D_n = \frac{(n-1)\kappa_{n-1}^2}{n\kappa_n\kappa_{n-2} - (n-1)\kappa_{n-1}^2}$. 
\end{conj}

}
\end{rem}

\section{Mixed Volumes and the Information Functional}  \label{Section-relation-with-entropy}

Let $K \in \K^n$. Recall that the information functional
$I$ is defined by  $I(K)=|K|/|\partial K|$, where $|K|$
stands for the volume of $K$, and $|\partial K|$ for its
surface area.
As mentioned in the introduction, the functional $I$ was
introduced and studied in~\cite{DCT}. In the same paper, it
was asked whether $I$ satisfies a Brunn--Minkowski type
inequality, i.e.,  whether for any $K_1,K_2 \in \K^n$ one has
$I(K_1+K_2) \geq I(K_1) + I(K_2)$, or a weaker property,
whether $I$ is monotone with respect to Minkowski addition,
namely, satisfies $I(K_1+K_2) \geq I(K_1)$, for every
$K_1,K_2 \in {\cal K}^n$.   
A counterexample to the above Brunn--Minkowski type
inequality was given in~\cite{FGM}. Although this was not
pointed out explicitly in~\cite{FGM}, this example is also
a counterexample for the monotonicity of the information
functional $I$ on the class ${\mathcal K}^n$.

We next turn to the proof of  Proposition~\ref{Prop_Equiv-To-Info-Monoton},
and in the next subsection we shall give another example
where both inequalities in Proposition~\ref{Prop_Equiv-To-Info-Monoton}
fail to hold. 


\begin{proof}[{\bf Proof of Proposition~\ref{Prop_Equiv-To-Info-Monoton}}]
Let $T\in\K^n$. Both directions of the equivalence follow from the fact that inequality ${(i)}$ is in a sense 
the ``linearization" of inequality ${(ii)}$. Indeed, assume ${(i)}$ holds. 
 For every $A\in \K^n$ and $\lambda\in[0,1]$, set $A_\lambda=A+
\lambda T$ and $f_A(\lambda)=I(A_\lambda)$. We wish to show that  $f_A(0)\le
f_A(1)$. To this end we shall prove that $f_A'\ge 0$ in the interval $[0,1]$. In fact, since $f_A(\lambda+h)=
f_{A_\lambda}(h)$, and $A$ is an arbitrary convex body, it is enough to prove
$f_A'(0)\ge 0$ for every $A\in\K^n$. Denote
${\cal V}_i=V(T[i],A[n-i])$ for $1\le i\le n$, and ${\cal W}_i=V(T[i],A[n-1-i],B_2^n)$
for $1\le i\le n-1$. In particular ${\cal V}_0=|A|,$ and $n{\cal W}_0= |\partial A|$. With these notation one has
\begin{eqnarray*}
 |A_\lambda | = V(A_\lambda,\ldots,A_\lambda)=\sum_{i=0}^n\binom{n}{i}
\lambda^i {\cal V}_i=|A|+n {\cal V}_1\lambda +o(\lambda),
\end{eqnarray*}
and similarly
\begin{eqnarray*}
 | \partial A_\lambda | = nV(A_\lambda,\dots,A_\lambda,B^n_2)=n\sum_{i=0}^{n-1} \binom{n-1}{i}\lambda^i {\cal W}_i
= | \partial A |+n(n-1){\cal W}_1\lambda+o(\lambda). 
\end{eqnarray*}
Combining these relations one has
\begin{eqnarray*}
I(A_\lambda)&=&I(A) \, \frac{1+n\frac{{\cal V}_1}{{\cal V}_0}\lambda+o(\lambda)}{1+(n-1)\frac{{\cal W}_1}{{\cal W}_0}\lambda+o(\lambda)}\\ 
& = & I(A) \left(1+\left(n \dfrac{{\cal V}_1}{{\cal V}_0}-(n-1)\frac{{\cal W}_1}{{\cal W}_0}\right)
\lambda+o(\lambda)\right).
\end{eqnarray*}
Thus, we conclude that \begin{equation} \label{eq-for-derivative-of-f} f_A'(0)=I(A)\cdot \left(
n\frac{{\cal V}_1}{{\cal V}_0}-(n-1)\dfrac{{\cal W}_1}{{\cal W}_0}\right). \end{equation}
Since inequality ${(i)}$ states exactly that ${\cal W}_0 {\cal V}_1\ge \frac{n-1}{n} {\cal W}_1{\cal V}_0$,  we conclude that $f'_A(0) \geq 0$, and hence inequality $(ii)$ holds.
Conversely, assume that inequality ${(i)}$ does not hold,  
that is, there exists $A\in\K^n$ such that ${\cal W}_0 {\cal V}_1< \frac{n-1}{n} {\cal W}_1{\cal V}_0$.
From~$(\ref{eq-for-derivative-of-f})$ it follows
that
$f_A'(0)<0$. From the continuity property of mixed volumes we conclude  that 
there exists $
\delta>0$ such that for all $\lambda\in[0,\delta]$, one has
$f_A'(\delta)<0$, and hence  
$f_A$ is strictly decreasing on $[0,\delta]$. From this
it follows that $I(A+\delta T)<I(A)$, or equivalently,
$I(\tilde{A}+T)<I(\tilde{A})$ for $\tilde{A}=A/\delta$. This
completes the proof of Proposition~\ref{Prop_Equiv-To-Info-Monoton}.
\end{proof}

\subsection{An example without monotonicity } \label{section-counter-example}

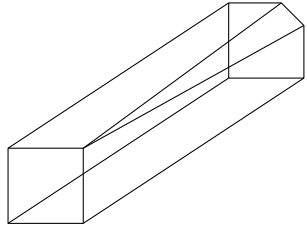
\begin{figure}[H]
\centering
\begin{tikzpicture} [line join=bevel,z=-5.5]
\coordinate (A1) at (1 ,1 ,10);  
\coordinate (A2) at (0 ,1 ,10);
\coordinate (A3) at (1 ,0 ,10);   
\coordinate (A4) at (0 ,0 ,10);

\coordinate (B12) at (1.7 ,1   ,0); 
\coordinate (B13) at (2   ,0.7 ,0);
\coordinate (B2) at (1   ,1 ,0);
\coordinate (B3) at (2   ,0 ,0);   
\coordinate (B4) at (1   ,0 ,0);

\draw (A1)           -- (A3) -- (A4) -- (A2) -- cycle; 
\draw (B12) -- (B13) -- (B3) -- (B4) -- (B2) -- cycle; 

\draw (A1)  -- (B12); 
\draw (A1)  -- (B13); 
\draw (A2)  -- (B2) ; 
\draw (A3)  -- (B3) ; 
\draw (A4)  -- (B4) ; 

\end{tikzpicture}
\caption{A counterexample to the monotonicity of the information functional $I$}
\label{fig:bla}
\end{figure}

In this subsection we provide a simple counterexample to the
monotonicity property of the functional $I$ for $n \geq 3$.
More precisely,  we give an example of a pair of convex
bodies $T,A \in K^n$ for which the two equivalent inequalities
in Proposition~\ref{Prop_Equiv-To-Info-Monoton} fail to hold.

\begin{prop} \label{prop-counter-exam}
If $n\ge 3$, then there exist $A,T\in\K^n$ such that $$I(A+T)<I(A).$$
\end{prop}

\begin{proof}[\bf Proof of Proposition~\ref{prop-counter-exam}]
It is enough to
find $A,T$ violating inequality $(i)$ of Proposition \ref{Prop_Equiv-To-Info-Monoton}. To do
that, we use an interval $T=[0,u]$, say for $u=e_n$. In that case,
inequality $(i)$ of Proposition \ref{Prop_Equiv-To-Info-Monoton}
becomes:
\begin{equation} \label{ineq-for-I-of-proj} I(A|_u)\ge I(A). \end{equation}
Here, to ease notation, we use $I$ to denote both the $n$-dimensional
information of $A$, and the $(n-1)$-dimensional information
of its projection to $u^\perp$, denoted by $A|_u$. Indeed,
inequality~$(\ref{ineq-for-I-of-proj})$ holds because $nV(B,T,A[n-2])=
\nu(B|_u,A|_u[n-2])$, $nV(T,A[n-1])=\nu(A|_u[n-1])$, and the
$(n-2)$-dimensional surface area of $A|_u$ is given by
${\rm Vol}_{n-2}(\partial A|_u)=(n-1)\nu(B|_u,A|_u[n-2])$.

For $A$ we shall take a long cylindrical body and cut out a
small piece of it by intersection with a half-space (see Figure 1),
in such a way that the projection $A|_u$ is unchanged, but
the information of $A$ will slightly increase.

More precisely, let $0<\varepsilon<1<M$, and denote by
$Q=\sum_{i=1}^{n-1}[0,e_i]$ the $(n-1)$-dimensional unit cube.
Let $D$ be the $(n+1)$-simplex with vertices
$0,\varepsilon e_1,\dots,\varepsilon e_{n-1},M e_n$, and
opposite facets $F_0,\dots,F_n$, respectively. Finally, set
$$A:=\left(Q+M[0,e_n]\right)\setminus D.$$ Note that $A|_u= Q$
and hence $I(A|_u)=\frac{1}{2(n-1)}$. Moreover, a direct
computation gives $|F_n|=\frac{\varepsilon^{n-1}}{(n-1)!}$,  $|F_i|=\frac{M\varepsilon^{n-2}}{(n-1)!}$ for $1\le i\le n-1$,
and 
\begin{eqnarray*} |F_0| & = &|\left({\rm diag}\{1,\dots,1,M/\varepsilon\}\right){\rm conv}\{\varepsilon e_i
\}_{i=1}^n|<\frac{M}{\varepsilon}|{\rm conv}\{\varepsilon e_i
\}_{i=1}^n| \\
& = & M\varepsilon^{n-2}|{\rm conv}\{e_i\}_{i=1}^n|= 
\frac{M\varepsilon^{n-2}\sqrt{n}}{(n-1)!}.\end{eqnarray*}

From this we conclude that
$$ |\partial(A)|=2M(n-1)+2- \sum_{i=1}^n|F_i|+|F_0|
<2M(n-1)+2-\frac{M\varepsilon^{n-2}}{(n-2)!}\left(1-\frac
{\sqrt{n}}{n-1}\right).$$ Since $|D|=\frac{M\varepsilon^{n-1}}{n!}$,
one has $$\frac{I(A)}{I(A|_u)}>\frac
{2M(n-1)\left(1-\frac{\varepsilon^{n-1}}{n!}\right)}
{2M(n-1)\left(
1+\frac{1}{M(n-1)}-\frac{\varepsilon^{n-2}}{2(n-1)!}\left(1-\frac
{\sqrt{n}}{n-1}\right)
\right)}.$$
Thus, by choosing $\varepsilon$ and $M$ such that:
$$ \frac{1}{M(n-1)}+\frac{\varepsilon^{n-1}}{n!}<
\frac{\varepsilon^{n-2}}{2(n-1)!}\left(1-\frac
{\sqrt{n}}{n-1}\right),$$ we have shown that $I(A)>I(A|_u)$.
This choice is indeed possible since for $n\ge 3$ the $\varepsilon^{n-2}$
coefficient is positive. The proof of Proposition~\ref{prop-counter-exam} is now complete.
\end{proof}

\section{Mixed Volumes in the Plane}\label{Sec_Vol-n=2}

In this section we show that in the $2$-dimensional case
inequality~$(\ref{MV-general-(false)-statement})$ for mixed
volumes always holds (cf.~\cite{BW} for a speical case). In
particular, it follows from this and Proposition
~\ref{Prop_Equiv-To-Info-Monoton} that for $n=2$ the information
functional $I$ is monotone with respect to Minkowski addition.

\begin{prop}\label{Thm_R2-General} Let $K,T,A\in\K^2$. Then 
\begin{equation}\label{main-ineq-in-dim2}
V(K,A)
V(T,A)\ge\frac{1}{2}V(K,T)V(A,A). \end{equation} Moreover,   equality holds
if and only if one of the following cases holds:  
(i) $K$ and $T$ are intervals, and $A$ is a parallelogram
whose edges are parallel to $K$ and $T$ 
(ii) $K$ and $A$, or $T$ and $A$ (or all) are contained in parallel intervals 
(iii) $A$ is a singleton. 
\end{prop}

We start with some preparations. Denote the inner and outer
radii of $T$ with respect to $A$ by $r=r_A(T)$, and $R=R_A(T)$
respectively. These are the optimal numbers satisfying that
$rA+x\subset T\subset RA+y$ for some $x,y\in\R^n$.  It is well
known that for $A,T\in\K^2$, the polynomial
$P(\lambda)=V(A,A)\lambda^2 + 2V(T,A)\lambda + V(T,T)$ has only
real roots (e.g., by Minkowski's inequality) which are clearly non-positive. Moreover, a
Bonnesen-type inequality (see pages 323-324 in~\cite{Sch}) states that
$$\label{Bonnesen-ineq} \lambda^-\le -R_A(T)\le -r_A(T)\le \lambda^+, $$
where $\lambda^{\pm}$ are the roots of $P(\lambda)$. In particular
$$P(-R)\le 0,\qquad P(-r)\le 0.$$
We are now in a position to prove Proposition~\ref{Thm_R2-General}.

\begin{proof}[\bf Proof of Proposition~\ref{Thm_R2-General}]
If $V(A,A)=0$, or $T$ is a singleton, we are done. Otherwise,
we have $0<R_A(T)<\infty$. Since the inequality is homogeneous
in each of the bodies, we may assume $R_A(T)=1$.
By the remark preceding the proof:
$$0\ge P(-1)=V(A,A)-2V(T,A)+V(T,T)\ge
V(A,A)-2V(T,A),$$ and hence
\begin{equation}\label{eq-2D-Bonnesen}
V(T,A) \geq \frac{1}{2}V(A,A) .
\end{equation}
Combining the latter with $V(K,A) \geq V(K,T)$ (which is due
to inclusion), proves $(\ref{main-ineq-in-dim2})$.
Next, we characterize the equality case. In the case where $V(A,A)=0$, either $A$ is a singleton, and there is equality, or $A$ is contained in an interval and at least one of $V(K,A)$,$V(K,T)$ must equal $0$, that is, we are in equality case (ii). 
Assume $V(A,A)>0$. A
necessary condition for equality in $(\ref{eq-2D-Bonnesen})$
is that $V(T,T)=0$, i.e. $T$ is contained in an interval. 
By symmetry, $K$ must be contained in an interval too. Next,
let $P$ be the minimal parallelogram containing $A$, which
has edges parallel to $K$ and $T$. Since $V(K,A)=V(K,P)$ and
$V(T,A)=V(T,P)$, we have:
$$V(K,A)V(T,A)=V(K,P)V(T,P)=
\frac{1}{2}V(T,K)V(P,P)\ge \frac{1}{2}V(T,K)V(A,A).$$
Equality holds above only if $V(P,P)=V(A,A)$, which implies
$A=P$.
\end{proof}
 
\begin{cor} \label{mon-of-I-in-dim2}
For every $A,T\in\K^2$, $I(A+T)>I(A)$.
\end{cor}
\begin{proof}[{\bf Proof of Corollary~\ref{mon-of-I-in-dim2}}]
Combining Proposition \ref{Prop_Equiv-To-Info-Monoton} with Proposition
\ref{Thm_R2-General} proves $I(A+T)\ge I(A)$. To show strict inequality
we only note that the derivative $f_A'$ is strictly positive, since
$K=B_2$ does not qualify for the equality case of Theorem
\ref{Thm_R2-General}.
\end{proof}

Finally, still in the $2$-dimensional case, we prove inequality~$(\ref{MV-general-(false)-statement})$  with an improved constant in the case where $A=B_2^2$ is the Euclidean disk. In fact, the following theorem amounts to Theorem \ref{Thm-about-mixed-volumes} in dimension 2 without the assumption that one of the bodies is a zonoid. 

\begin{prop}\label{Prop_2-dim-sharp-ineq}
Let $K,T\in\K^2$. Then: 
\begin{equation} \label{2-dim-case-A=B}
V(T,B^2_2)V(K,B^2_2)\ge\frac{2}{\pi}V(T,K)V(B^2_2,B^2_2)
,\end{equation} with equality if and only if $K$ and $T$ are orthogonal
intervals.
\end{prop}
\begin{proof}[{\bf Proof of Proposition~\ref{Prop_2-dim-sharp-ineq}}]
Denote by $R(T)$ the circumradius  of $T$ i.e., the smallest radius of a disc containing $T$, and by
$L(T)$ its perimeter.  Note that $ V(B^2_2,B^2_2)=\pi$, $V(T,B^2_2)=\frac{L(T)}{2}$, and from the monotonicity of mixed volume one has
$$ V(T,B^2_2)V(K,B^2_2) \geq  {\frac { L(T) V(K,T)} {2R(T)}}.$$
Inequality~$(\ref{2-dim-case-A=B})$  now follows from that fact  that  $L(T) \geq 4R(T)$ (see \cite{Nit}).
Moreover, $L(T)=4R(T)$ if and only if $T$ is an interval. 
Since $K$ and $T$ play a symmetric role in~$(\ref{2-dim-case-A=B})$, a 
necessary condition for equality in~$(\ref{2-dim-case-A=B})$ is that both $K$ and $T$ are intervals.
In that case, a direct computation shows that equality holds if and only if $K$ and $T$ are orthogonal. 
The  proof of  Proposition~\ref{Prop_2-dim-sharp-ineq} is now complete.
\end{proof}

\noindent Shiri Artstein-Avidan, Dan Florentin, Yaron Ostrover\\
\noindent School of Mathematical Science, Tel Aviv University, Tel Aviv, Israel.\\
\noindent {\it e-mails}: shiri@post.tau.ac.il, 
 danflore@post.tau.ac.il ,
  ostrover@post.tau.ac.il.

\end{document}